\documentclass[english,11pt]{article} \usepackage[latin1]{inputenc} \usepackage{amsmath,amsthm,amssymb,babel}

\newtheorem{theorem}{Theorem}
 \newtheorem{lemma}{Lemma}
  \newtheorem{definition}{Definition}
   
    \newtheorem{remark}{Remark}
\newtheorem{proposition}{Proposition}
 \newtheorem{problem}{Problem}

    \newcommand{\bnu}{\mbox{\boldmath{$\nu$}}}

  \newcommand{\bx}{\mbox{\boldmath{$x$}}}

  \date{ }
\begin{document}

\title{On the relationship
between alternative variational \\formulations of a frictional contact model}

          \author{Andaluzia Matei\footnote{This is a preprint. To cite the final version: A. Matei, On the relationship between alternative variational formulations of a frictional contact model, https://doi.org/10.1016/j.jmaa.2019.123391, Journal of Mathematical Analysis and Applications, Volume 480, Issue 1, 1 December 2019, 123391} \\Department of Mathematics, University of Craiova\\
 A.I. Cuza 13, 200585, Craiova, Romania\\ E-mail: andaluziamatei@inf.ucv.ro}

 \maketitle

\begin{abstract} We consider a frictional contact model, mathematically described by means of a nonlinear boundary value problem in terms of PDE. We draw the attention to three possible variational formulations of it.  One of the variational formulations is a variational inequality of the second kind and the other two are mixed variational formulations with Lagrange multipliers in dual spaces. As main novelty, we establish the relationship between these three variational formulations. We also pay attention to the recovery of the formulation in terms of PDE starting from the mixed variational formulations.
\end{abstract}

\noindent\textbf{Key words:}
   boundary value problem, frictional contact, variational inequality of the second kind, mixed variational formulation, dual Lagrange multipliers, weak solutions.

 \noindent\textbf{AMS subject classification}: 35J65, 49J40, 74M10, 74M15.

\section{Introduction}


A large variety of real world phenomena involves bodies in interaction. Such phenomena are described by means of the contact models, see, e.g., \cite{L13,KW,WRIG}. From the mathematical point of view, the contact models are nonlinear problems
described by differential equations with nonlinear boundary conditions, see \cite{DL} or, more recently, see for instance \cite{HS02,SST}. The formulations of the contact phenomena by means of the boundary value problems are called strong formulations; the solutions of the strong formulations are called strong solutions.  Due to their complexity, generally, the contact problems don't have strong solutions. Thus, in the analysis of the contact models under less restrictive conditions on the data, the concept of weak solution is crucial. The interest on the variational approach is due to computational reasons.

For a frictional contact model we can deliver several variational formulations.
It is well known, for instance, that the bilateral frictional contact problems lead to {\it variational inequalities of the second kind}. At the same time such models lead to variational systems with Lagrange multipliers called {\it mixed variational formulations}.
Each solution of a variational formulation will be called a weak solution of the model. So, we can define several weak solutions for a model.

What is the relationship between the weak solutions corresponding to different variational formulations? Is it possible to recover the strong formulation of the model starting from the mixed variational formulations? These two questions will be answered in the present paper by means of an experiment  involving an antiplane frictional contact model.
Precisely, we consider the following boundary value problem.
\begin{problem} \label{bvp}
Find $u:\bar{\Omega} \to \mathbb R$ such that
\begin{eqnarray}
\label{c1} -\xi\Delta u(\bx) =f_0(\bx) \qquad\qquad\quad\qquad\qquad \mbox{\^{\i}n } \Omega, \\[2mm]
u(\bx)=0 \qquad\qquad\quad\qquad \qquad\,\,\,\,\, \mbox{ pe } \Gamma_1, \label{c1bis}\\[2mm]
\xi \frac{\partial u}{\partial \nu}(\bx) = f_2(\bx) \qquad\qquad\,\,\,\,\,\,\,\qquad\quad \mbox{ pe } \Gamma_2,\label{c2} \\[2mm]
\begin{array}{c}\vspace{2mm}
\left|\xi\displaystyle\frac{\partial u}{\partial \nu}(\bx)\right|\leq g,\quad
\xi\displaystyle\frac{\partial u}{\partial \nu}(\bx)=-g\frac{u(\bx)}{|u(\bx)|}\quad \mbox{ if }  u(\bx)\neq 0
\end{array}\qquad\,\, \,\,\,\, \mbox{ pe } \Gamma_3.\label{c3}
\end{eqnarray}
\end{problem}
Problem \ref{bvp} is a boundary value problem with nonlinear boundary conditions. The formulation (\ref{c1})-(\ref{c3}) is called the strong formulation of Problem \ref{bvp}. Herein $\Omega \subset \mathbb R^2$   is a bounded domain with smooth boundary $\Gamma,$ partitioned in three measurable parts with positive measure, $\Gamma_i$, $i\in\{1,2,3\},$
$g>0,$  $\xi>0,$ $f_0 \in L^2(\Omega),$ $f_2 \in L^2(\Gamma_2),$ and, as usual, $\frac{\partial u}{\partial \nu}=\nabla u\cdot \bnu,$ $\bnu$ denoting the unit outward normal vector at $\Gamma$. Problem \ref{bvp} models the frictional contact between a linearly elastic cylindrical body and a rigid foundation in the antiplane context, under the small deformation hypotheses; see, e.g., \cite{SM09} for details on the frictional antiplane models in contact mechanics.

Three possible variational formulations of Problem \ref{bvp} are described in the present paper. One of them, a variational inequality of the second kind, is a classical one. Another one is a variational formulation by means of Lagrange multipliers in the dual of a subspace of $H^{1/2}(\Gamma).$ As it is known, the weak solution defined via variational inequalities of the second kind can be approximated by means of a regularization technique. In contrast, the weak solution defined via mixed variational formulations with Lagrange multipliers in the dual of a subspace of $H^{1/2}(\Gamma)$ can be approximated directly, by means of  modern numerical techniques, like the primal-dual active set strategy, see, e.g., \cite{HMW07} and the references therein. In addition to the aforementioned variational formulations, in the present paper we deliver a new mixed variational formulation which is interesting in its own but it is also helpful to establish a connection between the mixed variational formulations and the classical variational formulation.
After we study the relationship between the three weak formulations, assuming enough smoothness, we discuss the recovery of the strong formulation starting from the mixed variational formulations.

The present study is important for the numerical treatment of the model. Our analysis confirms that, if we approximate the solutions of the mixed variational formulations, then we really approximate the solutions of the mechanical model.

The rest of the paper has the following structure. Section 2 contains some preliminaries. Section 3  address  three variational formulations of Problem \ref{bvp}, discussing the solvability in every case. In Section 4 we analyze the relationship between the  variational formulations we present in Section 3. The last section is devoted to the recovery of the strong formulation assuming enough smoothness of the data and weak solutions.

\section{Preliminaries}
Let $(X, (\cdot,\cdot)_X, \|\cdot\|_X)$ and $(Y, (\cdot,\cdot)_Y, \|\cdot\|_Y)$ be  two Hilbert spaces.
\begin{definition} A bilinear form $e:X\times Y\to \mathbb R$ is continuous (of rank $M_e$) if there exists $M_e>0$ such that
\begin{equation*}
|e(v,\mu)|\leq M_e\|v\|_X\|\mu\|_{Y}\, \mbox{ for all } v\in X,\,\mu\in Y.
\end{equation*}
\end{definition}
\begin{definition} A bilinear form $e:X\times X\to \mathbb R$ is $X-$elliptic (of rank $m_e$) if there exists $m_e>0$ such that
\begin{equation*}
e(u,u)\geq m_e\|u\|^2_X\quad \mbox{ for all }u\in X.
\end{equation*}
\end{definition}

Let us consider the following mixed variational problem.

\begin{problem}\label{ap} Find $u\in X$ and $\lambda\in \Lambda$
such that
\begin{eqnarray}
   a(u,v)+b(v,\lambda)&=&(f,\,v)_X\qquad\,
    \mbox{ for all } v\in X,\label{l1}\\
b(u,\mu-\lambda)&\leq&0\qquad\qquad\,\,\,\,\mbox{ for all } \mu\in
\Lambda,\label{l2}
 \end{eqnarray}
\end{problem}
where:
\begin{eqnarray}\label{iA1}
&&\bullet\,\,\, a:X\times X \to \mathbb R \mbox{ is a symmetric bilinear continuous (of rank $M_a$)}\nonumber\\
&&\mbox{and $X-$elliptic (of rank $m_a$) form;}\\
&& \bullet\,\,\, b:X\times Y\to {R} \mbox{ is a bilinear continuous (of rank $M_b$) form,}\label{ib1}\\
 &&\mbox{and, in addition,}\nonumber\\[2mm]
&&\exists\, \alpha>0:\displaystyle\inf_{\mu\in Y,
  \mu\neq 0_{Y}}\,\sup_{v\in X,v\neq 0_X}\,\frac{b(v,\mu)}{\|v\|_X\|\mu\|_{Y}}\geq \alpha;\label{ib3}\\
&&\bullet\,\,\, \Lambda \text{ is a closed convex subset of } Y \text{
that contains }0_{Y}.\label{iL}
\end{eqnarray}
The following existence and uniqueness result takes place.

\begin{theorem}\label{abs} Assume
$(\ref{iA1})$--$(\ref{iL}).$  Then, Problem $\ref{ap}$ has a unique solution $(u,\lambda)$  in $X\times \Lambda.$
\end{theorem}

The proof is based on the saddle point theory; see, e.g., \cite{ET,HHN96}.

Let us associate to the form $b(\cdot,\cdot)$ the linear and continuous operator $B:X\to Y'$ defined as follows: for each $v\in X,$
\begin{equation}\label{B}
\langle Bv,\lambda\rangle_{Y',Y}=b(v,\lambda)\quad\mbox{ for all }\lambda\in Y.
\end{equation}
Moreover, we can associate to $B$ the linear and continuous operator $B^t:Y\to X'$ defined as follows: for each $\lambda\in Y,$
\begin{equation}\label{Bt}
\langle B^t\lambda,v\rangle_{X',X}=\langle Bv,\lambda\rangle_{Y',Y}\quad\mbox{ for all }v\in X.
\end{equation}
Actually, we can write
\begin{equation}\label{yut}
\langle B^t\lambda,v\rangle_{X',X}=\langle Bv,\lambda\rangle_{Y',Y}=b(v,\lambda)\quad\mbox{ for all }v\in X, \lambda\in Y,
\end{equation}
see \cite{BBF13} (pages 210-213) and  \cite{Braess}(page 131).

According to, e.g., \cite{BBF13}(see 4.1.61-4.1.62), keeping in mind (\ref{ib1})-(\ref{ib3}), we have to write
\begin{equation}\label{imp}
Im B^t=(Ker B)^0,
\end{equation}
where, as it is known,
\[Ker B=\{v\in X\,|\, Bv=0\}\]
and
$(Ker B)^0$ denotes, as usual, the polar of $Ker B,$ i.e.,
\begin{eqnarray*}
(Ker B)^0=\{ l\in X'\,|\, \langle l,v\rangle_{X',X}=0\mbox{ for all } v\in Ker B\}.
\end{eqnarray*}
Notice that, (\ref{yut}) yields
\begin{eqnarray}
Ker B&=&\{v\in X\,|\, b(v,\lambda)=0\mbox{ for all }\lambda\in Y\}\label{rB1}\\
&=&\{v\in X\,|\,\langle B^t\lambda,v\rangle_{X',X} =0\mbox{ for all }\lambda\in Y\};\nonumber
\end{eqnarray}
see, e.g., (4.1.52) in \cite{BBF13}.

Moreover,
\begin{equation}\label{izo}
B^t:Y\to (Ker B)^0\subset X'\mbox{ is an isomorphism};
\end{equation}
see, e.g., Theorem 3.6 page 125 and Lemma 4.2 page 131 in \cite{Braess}.

\section{Weak formulations}
In this section we present three possible variational formulations of Problem \ref{bvp}.
\subsection{First variational formulation}
In this first subsection we recall a variational formulation of Problem \ref{bvp} in terms of variational inequalities of the second kind.
\begin{problem}\label{vf1} Find $u_0\in X$ such that
\begin{equation}\label{l:vf1}
a(u_0,v-u_0)+j(v)-j(u_0)\geq (f,v-u_0)_X\quad\mbox{ for all }v\in X,
\end{equation}
\end{problem}
where
\begin{eqnarray}
&X=\{v \in H^1(\Omega) \ |  \ \gamma v=0 \mbox{ a.e. on } \Gamma_1\};\quad  (u, v)_X = ( \nabla u, \nabla v)_{L^2(\Omega)^2};&\label{X}\\
&a:X\times X\to \mathbb R\quad a(u,v)=\xi(u,v)_X;&\label{a}\\[2mm]
& j:X\to \mathbb R\quad j(v)=\int_{\Gamma_3}g|\gamma v|\,d\Gamma;&\label{j}\\[2mm]
& (f,v)_X=\int_{\Omega}f_0\,v\,dx+\int_{\Gamma_2}f_2\,\gamma v\,d\Gamma.&\label{f}
\end{eqnarray}
Everywhere in this paper $\gamma:H^1(\Omega)\rightarrow L^2(\Gamma)$ is the trace operator.

According to the theory of the variational inequalities of the second kind, the following result takes place:
\begin{theorem}\label{t:1}
 Problem \ref{vf1} has a unique solution $u_0\in X.$
\end{theorem}
For details, see, e.g., \cite{SM09} and the references therein.

\subsection{Second variational formulation}
 Let us recall a variational formulation of Problem \ref{bvp} in terms of Lagrange multipliers. In addition to the Hilbert space $X,$ given by (\ref{X}), we need in this subsection the following space,
\begin{eqnarray*}
M=\gamma(X)=\{\widetilde{v}\in H^{1/2}(\Gamma)|\,\,\,\mbox{there exists}\,\,\,v\in X \,\,\,\mbox{such that}\,\,\,\widetilde{v}=\gamma\,v \,\,\mbox{a.e. on}\,\,\Gamma\},\label{S}
\end{eqnarray*}
where, recall,
\begin{equation*}
H^{1/2}(\Gamma)=\{\widetilde{v}\in L^2(\Gamma)|\,\,\,\mbox{there exists}\,\,\,v\in H^1(\Omega) \,\,\,\mbox{such that}\,\,\,\widetilde{v}=\gamma\,v \,\,\mbox{a.e. on}\,\,\Gamma\}.
\end{equation*}
 Let us denote by $\mathcal Z:H^{1/2}(\Omega)\to H^1(\Omega)$ the right inverse of the trace operator. As it is known, $\mathcal Z$ is a linear and continuous operator. It is worth to emphasize that
\begin{eqnarray*}
{\mathcal Z}(\gamma v)\in X \quad\mbox{ for all }v\in X,
\end{eqnarray*}
and
\begin{eqnarray}\label{zzz}
\gamma ({\mathcal Z}(\gamma v))=\gamma v\quad\mbox{ for all }v\in X.
\end{eqnarray}
For details on the space $H^{1/2}(\Gamma)$, its structure and related operators, see, e.g., \cite{KJF,Necas}.

We also need the dual of the space $M$ denoted by $Y,$
\begin{equation}\label{Y}
Y=M'.
\end{equation}
The space $M$ is a Hilbert space, see, e.g., \cite{MC}. Furthermore, $Y$ is a Hilbert space being the dual of the Hilbert space $M.$ Everywhere below
$\langle\cdot,\cdot\rangle_{Y,M}$ will denote the dual pairing between $Y$ and $M.$

For a regular enough function $u$ which verifies Problem \ref{bvp}, a Lagrange multiplier $\lambda\in Y$
can be introduced as follows:
\begin{eqnarray*}
\langle \lambda,\widetilde{v} \rangle_{Y,M}=-\int_{\Gamma_3} \xi\frac{\partial u}{\partial \nu}\,\widetilde{v}\,d\Gamma\quad\mbox{ for all }\widetilde{v}\in M.
\end{eqnarray*}
The first mixed variational formulation of Problem \ref{bvp} is the following.
\begin{problem} \label{vf2}
Find $u\in X$ and $\lambda\in \Lambda\subset Y$ such that
\begin{eqnarray}
a(u,v)+\langle\lambda,\gamma v\rangle_{Y,M}&=&(f,v)_X \quad\qquad\mbox{for all}\,\,\,\, v\in X,    \label{l1:vf2}\\
\langle \mu-\lambda,\gamma u\rangle_{Y,M}&\leq& 0\,\qquad\qquad\quad\,\mbox{for all}\,\, \,\,\mu \in \Lambda.\label{l2:vf2}
\end{eqnarray}
\end{problem}
For the definition of  $a(\cdot,\cdot)$  and $f$ see (\ref{a}) and (\ref{f}), respectively. The set of the Lagrange multipliers $\Lambda$ in Problem \ref{vf2} has the following definition.
\begin{eqnarray}\label{L}
\Lambda=\{\mu\in Y\,|\, \langle \mu,\gamma v\rangle_{Y,M}\leq \int_{\Gamma_3}g|\gamma v|\,d\Gamma\mbox{ for all }v\in X\}.
\end{eqnarray}

It is easy to observe that (\ref{iA1}) and (\ref{iL}) are fulfilled.

Let us define
\begin{eqnarray}
 b:X\times Y\to \mathbb R\quad b(v,\mu)=\langle \mu,\gamma v\rangle_{Y,M}\quad\mbox{ for all }v\in X,\,\mu\in Y.\label{b}
\end{eqnarray}
Obviously, the form $b(\cdot,\cdot)$ is a bilinear and continuous form. In addition,
\begin{eqnarray*}
\|\mu\|_Y=\sup_{\overline{w}\in M,\,\overline{w}\neq 0_M}\frac{\langle \mu,\overline{w}\rangle_{Y,M}}{\|\overline{w}\|_M}
\leq \sup_{w\in X,\,\gamma w\neq 0_M}\frac{\langle \mu,\gamma w\rangle_{Y,M}}{\|\gamma w\|_M}.
\end{eqnarray*}

Let $w\in X$ be such that $\gamma w\neq 0_M.$ Then ${\mathcal Z}(\gamma w)\neq 0_X.$ Indeed, otherwise ${\mathcal Z}(\gamma w)=0_X$, and applying the trace operator $\gamma$ we obtain $\gamma({\mathcal Z}(\gamma w))=\gamma 0_X=0_M.$ Using now (\ref{zzz}) it results that $\gamma w=0_M$ which is not true.
Because ${\mathcal Z}$ is linear and continuous, there exists $c>0$ such that,
\[\|{\mathcal Z}(\gamma w)\|_X\leq c\|\gamma w\|_M.\]
As $\gamma w\neq 0_M,$ then
\[\frac{1}{\|\gamma w\|_M}\leq\frac{c}{\|{\mathcal Z}(\gamma w)\|_X},\] the right hand side of this inequality being well defined.

Therefore,
\begin{eqnarray*}
\|\mu\|_Y&\leq& c\sup_{w\in X,\,\mathcal Z(\gamma w)\neq 0_X}\frac{\langle \mu,\gamma w\rangle_{Y,M}}{\|{\mathcal Z}(\gamma w)\|_X}\\
  &=&c\sup_{w\in X,\,\mathcal Z(\gamma w)\neq 0_X}\frac{\langle \mu,\gamma({\mathcal Z}(\gamma w))\rangle_{Y,M}}{\|{\mathcal Z}(\gamma w)\|_X}\\
  &=&c\sup_{w\in X,\,\mathcal Z(\gamma w)\neq 0_X}\frac{b({\mathcal Z}(\gamma w),\mu)}{{\|\mathcal Z}(\gamma w)\|_X}\\
  &\leq& c \sup_{v\in X,\,v\neq 0_X}\frac{b(v,\mu)}{\|v\|_X}.
\end{eqnarray*}
By consequence (\ref{ib1})-(\ref{ib3}) hold true.

The following result is a direct application of Theorem \ref{abs}.
\begin{theorem}\label{t:2} Problem \ref{vf2} has a unique solution $(u,\lambda)\in X\times \Lambda.$
\end{theorem}

From the numerical point of view, Problem \ref{vf2} is a very convenient variational formulation, see, e.g., \cite{HMW07}.

\subsection{Third variational formulation}
Let us draw the attention on a new variational formulation via Lagrange multipliers. Everywhere in this subsection we denote by $X$
the space given by (\ref{X}), by $X'$ its dual and by $\langle\cdot,\cdot\rangle_{X',X}$ the dual pairing.

Let $\bar{u}$ be a regular enough function which verifies Problem \ref{bvp}.
We introduce a Lagrange multiplier $\bar{\lambda}\in X'$
as follows:
\begin{eqnarray*}
\langle \bar{\lambda},v \rangle_{X',X}=-\int_{\Gamma_3}\xi \frac{\partial \bar{u}}{\partial \nu} \,\gamma v\,d\Gamma\quad\mbox{ for all }v\in X.
\end{eqnarray*}
We deliver  the following mixed variational formulation.
\begin{problem} \label{vf3}
Find $\bar{u}\in X$ and $\bar{\lambda}\in \bar{\Lambda}\subset X'$ such that
\begin{eqnarray}
a(\bar{u},v)+\langle\bar{\lambda},v\rangle_{X',X}&=&(f,v)_X \quad\qquad\mbox{for all}\,\,\,\, v\in X,    \label{l1:vf3}\\
\langle \bar{\mu}-\bar{\lambda},\bar{u}\rangle_{X',X}&\leq& 0\,\qquad\qquad\quad\,\mbox{for all}\,\, \,\,\bar{\mu} \in \bar{\Lambda}\label{l2:vf3}.
\end{eqnarray}
\end{problem}

Again, $a(\cdot,\cdot)$ is given by (\ref{a}) and $f$ is given by (\ref{f}). Obviously, (\ref{iA1}) holds true.
The subset $\bar{\Lambda}$ in Problem \ref{vf3}  has the following definition:
\begin{eqnarray}\label{Lbis}
\bar{\Lambda}=\{\bar{\mu}\in X'\,|\, \langle \bar{\mu},v\rangle_{X',X}\leq \int_{\Gamma_3}g|\gamma v|\,d\Gamma\mbox{ for all }v\in X\}.
\end{eqnarray}
 Clearly, $\bar{\Lambda}$ in (\ref{Lbis}) is a closed convex subset of $X'$ containing $0_{X'},$ so (\ref{iL}) is fulfilled.

Now, we define $\bar{b}:X\times X'\to \mathbb R$ as follows,
\begin{eqnarray}
\bar{b}(v,\bar{\mu})=\langle \bar{\mu}, v\rangle_{X',X}\quad\mbox{ for all }v\in X,\,\bar{\mu}\in X'.\label{bbis}
\end{eqnarray}
This is a bilinear continuous form. In addition,
\begin{eqnarray*}
\|\bar{\mu}\|_X=\sup_{v\in X,\,v\neq 0_X}\frac{\langle \bar{\mu},v\rangle_{X',X}}{\|v\|_X}
 =\sup_{v\in X,\,v\neq 0_X}\frac{\bar{b}(v,\bar{\mu})}{\|v\|_X}.
\end{eqnarray*}
 Clearly, $\bar{b}$ fulfills (\ref{ib1})-(\ref{ib3}).

Applying Theorem \ref{abs} we obtain the following existence and uniqueness result.

\begin{theorem}\label{t:3} Problem \ref{vf3} has a unique solution $(\bar{u},\bar{\lambda})\in X\times \bar{\Lambda}.$
\end{theorem}
Problem \ref{vf3} will play a crucial role in the next section.

\section{On the relationship between the weak formulations}

In this section we focus on the relationship between the variational formulations presented in Section 3.

Let $u_0\in X$ be the unique solution of Problem \ref{vf1}. We define $\lambda_0\in X'$ by means of the following relation.
\begin{equation}\label{bl}
\langle \lambda_0,v\rangle_{X',X}=(f,v)_X-a(u_0,v)\quad\mbox{ for all }v\in X.
\end{equation}

\begin{theorem}\label{r:1} Let $(\bar{u},\bar{\lambda})\in X\times \bar{\Lambda}$ be the unique solution of Problem \ref{vf3},
let $u_0$ be the unique solution of Problem \ref{vf1} and let $\lambda_0\in X'$  be defined in (\ref{bl}).
Then,

\[\bar{u}=u_0\,\,\,\mbox{ and }\,\,\,\bar{\lambda}=\lambda_0.\]
\end{theorem}
\begin{proof}
Since $u_0\in X$ is the unique solution of Problem \ref{vf1}, then
\begin{eqnarray}
a(u_0,u_0)+j(u_0)=(f,u_0)_X;\label{eguti}\\
a(u_0,v)+j(v)\geq (f,v)_X\quad\mbox{ for all }v\in X.\label{iuti}
\end{eqnarray}
By (\ref{bl}) and (\ref{iuti}) we obtain that
\begin{eqnarray}\label{34}
\langle \lambda_0,v\rangle_{X',X}\leq j(v)\quad\mbox{ for all }v\in X.
\end{eqnarray}
Therefore,
\begin{equation}\label{rel1}
\lambda_0\in \bar{\Lambda}
\end{equation}
where $\bar{\Lambda}$ was defined by (\ref{Lbis}).
Furthermore, by (\ref{bl}) we obtain
\begin{equation}\label{rel2}
a(u_0,v)+\langle \lambda_0,v\rangle_{X',X}=(f,v)_X\quad\mbox{ for all }v\in X.
\end{equation}
On the other hand, due to (\ref{Lbis}), for all $\bar{\mu}\in \bar{\Lambda},$
\begin{eqnarray}\label{q2}
\langle \bar{\mu},u_0\rangle_{X',X}\leq j(u_0),
\end{eqnarray}
and due to (\ref{eguti}) and (\ref{bl}) we have
\begin{eqnarray}\label{q1}
\langle \lambda_0,u_0\rangle_{X',X}=j(u_0).
\end{eqnarray}
By (\ref{q2}) and (\ref{q1}) we obtain
\begin{eqnarray}\label{rel3}
\langle \bar{\mu}-\lambda_0,u_0\rangle_{X',X}\leq 0\quad \mbox{ for all }\bar{\mu}\in \bar{\Lambda}.
\end{eqnarray}
By (\ref{rel1}), (\ref{rel2}) and (\ref{rel3}) we obtain that $(u_0,\lambda_0)$ is a solution of Problem \ref{vf3}.
But, according to Theorem \ref{t:3}, Problem \ref{vf3} has a unique solution $(\bar{u},\bar{\lambda})\in X\times \bar{\Lambda}$. The conclusion is now immediate.
\end{proof}

According to (\ref{B}) and (\ref{Bt}), we can associate to the form $b(\cdot,\cdot)$ introduced in (\ref{b}), two linear and continuous operators $B:X\to Y'$ and $B^t:Y\to X',$ such that, for all $v\in X,$ $\lambda\in Y,$
\begin{equation}\label{uty}
\langle B^t\lambda,v\rangle_{X',X}=\langle Bv,\lambda\rangle_{Y',Y}=b(v,\lambda)=\langle\lambda,\gamma v\rangle_{Y,M}.
\end{equation}
\begin{lemma}\label{lemau}
$Ker B=Ker \gamma$.
\end{lemma}
\begin{proof}
Let $v\in Ker B.$ By (\ref{rB1}) we have
\[b(v,\lambda)=0\quad\mbox{ for all }\lambda\in Y.\]
Consequently,
\[\langle\lambda,\gamma v\rangle_{Y,M}=0\quad\mbox{ for all }\lambda\in Y.\]
Then, $\gamma v=0_M.$ Indeed, for $\lambda=R_{M}\gamma v$ we can write
\[0=\langle\lambda,\gamma v\rangle_{Y,M}=\|\gamma v\|_M^2.\]
Herein $R_M:M\to M'$ is the operator which appear in Ritz's Theorem,
see Theorem 4.1.2 page 209 in \cite{BBF13}.
Thus,
\begin{equation}\label{@}
Ker B\subseteq Ker \gamma.
\end{equation}

Let $v\in Ker \gamma.$ Then $\gamma v=0_M.$ Therefore $\langle\lambda,\gamma v\rangle_{Y,M}=0$ for all $\lambda\in Y=M'.$
According to (\ref{uty}), we can write $\langle Bv,\lambda\rangle_{Y',Y}=0$ for all $\lambda\in Y.$ Consequently, $Bv=0_{Y'}.$
So, $v\in Ker B.$ Therefore,
\begin{equation}\label{@@}
Ker \gamma\subseteq Ker B.
\end{equation}
The conclusion is a straightforward consequence of the inclusions (\ref{@}) and (\ref{@@}) .
\end{proof}

\begin{proposition}Let $(\bar{u},\bar{\lambda})\in X\times \bar{\Lambda}$ be the unique solution of Problem \ref{vf3}. Then
\begin{equation}\label{ap2}
\bar{\lambda}\in (Ker B)^0.
\end{equation}
\end{proposition}
\begin{proof}
Keeping in mind (\ref{j}), we observe that  $j(v)=j(-v)$ for all $v\in X.$ Therefore,
\begin{eqnarray}
|\langle\bar{\lambda},v\rangle_{X',X}|\leq j(v)\quad\mbox{ for all }v\in X.
\end{eqnarray}
And from this, since $j(v)=0$ for all $v\in Ker \gamma,$ we obtain that
\[\langle\bar{\lambda},v\rangle_{X',X}=0\mbox{ for all }v\in Ker \gamma.\]
By Lemma \ref{lemau} it follows that
\[\langle\bar{\lambda},v\rangle_{X',X}=0\mbox{ for all }v\in Ker B\]
i.e. we have (\ref{ap2}).
 \end{proof}

According to (\ref{izo}), there exists a unique $\widetilde{\lambda}\in Y$ such that
\begin{equation}\label{ll}
B^t\widetilde{\lambda}=\bar{\lambda}.
\end{equation}
\begin{theorem}\label{r:2} Let $(\bar{u},\bar{\lambda})\in X\times \bar{\Lambda}$ be the unique solution of Problem \ref{vf3}, let $\widetilde{\lambda}$ given by (\ref{ll}) and let $(u,\lambda)\in X\times \Lambda$ be the unique solution of Problem \ref{vf2}. Then,
 \[\bar{u}=u\,\,\,\mbox{ and }\,\,\,\widetilde{\lambda}=\lambda.\]
\end{theorem}
\begin{proof}
Due  to (\ref{ll}) and (\ref{uty}),
\begin{equation}\label{nr2}
\langle \bar{\lambda},v\rangle_{X',X}= \langle B^t\widetilde{\lambda},v\rangle_{X',X}=\langle Bv,\widetilde{\lambda}\rangle_{Y',Y}=b(v,\widetilde{\lambda})=\langle\widetilde{\lambda},\gamma v\rangle_{Y,M}.
\end{equation}
As $\bar{\lambda}\in\bar{\Lambda},$ we conclude easily that $\widetilde{\lambda}$ is an element of $\Lambda$; see (\ref{L}) and (\ref{Lbis}).
Moreover, due to (\ref{l1:vf3}) and (\ref{nr2}),
\[a(\bar{u},v)+\langle\widetilde{\lambda},\gamma v\rangle_{Y,M}=(f,v)_X\quad\mbox{ for all }v\in X.\]
On the other hand, for all $\mu\in \Lambda,$ we have
\begin{equation}\label{k1}
\langle \mu,\gamma \bar{u}\rangle_{Y,M}\leq j(\bar{u}).
\end{equation}
Recall that, according to (\ref{q1}), as $u_0=\bar{u}$ and $\lambda_0=\bar{\lambda},$ see Theorem \ref{r:1}, then
\begin{equation}\label{nr3}
\langle\bar{\lambda},\bar{u}\rangle_{X',X}=j(\bar{u}).
\end{equation}
By (\ref{nr2}) and (\ref{nr3}),
\begin{equation}\label{k2}
\langle \widetilde{\lambda},\gamma \bar{u}\rangle_{Y,M}=\langle B^t\widetilde{\lambda},\bar{u}\rangle_{X',X}=\langle \bar{\lambda},\bar{u}\rangle_{X',X}= j(\bar{u}).
\end{equation}
By (\ref{k1}) and (\ref{k2}) we obtain
\[\langle\mu-\widetilde{\lambda}, \gamma \bar{u}\rangle_{Y,M}\leq 0 \quad\mbox{ for all }\mu\in \Lambda.\]
Thus, $(\bar{u},\widetilde{\lambda})\in X\times \Lambda$ verifies Problem \ref{vf2}. As Problem \ref{vf2} has a unique solution $(u,\lambda)\in X\times \Lambda,$ we conclude this theorem.
\end{proof}
\begin{remark} Let $u_0$ be the unique solution of Problem \ref{vf1},  $(u,\lambda)\in X\times \Lambda$ be the unique solution of Problem \ref{vf2} and $(\bar{u},\bar{\lambda})\in X\times \bar{\Lambda}$ be the unique solution of Problem \ref{vf3}.
Then,
\[u_0=u=\bar{u}\]
and
\[B^t\lambda=\bar{\lambda},\]
where $B^t$ is the operator defined by (\ref{uty}). Moreover,
\begin{equation}\label{ecuy}
\langle \lambda,\gamma v\rangle_{Y,M}=\langle \bar{\lambda},v\rangle_{X',X}\leq j(v)\quad\mbox{ for all } v\in X,
\end{equation}
 \begin{eqnarray}\label{i3}
 \langle\lambda,\gamma u\rangle_{Y,M}=\langle\bar{\lambda},u\rangle_{X',X}=j(u)=\int_{\Gamma_3}g|\gamma u|\,dx,
 \end{eqnarray}
and
\begin{equation}\label{ecuyy}
a(u,v)+\langle \bar{\lambda},\gamma v\rangle_{X',X}=(f,v)_X \quad\mbox{ for all } v\in X.
\end{equation}
\end{remark}

\section{Recovery of the strong formulation}

According to the study in the previous sections, Problems \ref{vf2} and \ref{vf3} are mixed variational formulations of Problem \ref{bvp}, each of them having a unique solution $(u,\lambda)\in X\times \Lambda$ and $(u,\bar{\lambda})\in X\times \bar{\Lambda}$, respectively.

How the weak solutions  $(u,\lambda)\in X\times \Lambda$ and $(u,\bar{\lambda})\in X\times \bar{\Lambda}$ are related to Problem \ref{bvp}?

Let us assume everywhere below that the data and the weak solutions are smooth enough. First of all we notice that $u\in X$ implies $u(\bx)=0$ on $\Gamma_1,$ so (\ref{c1bis}) is recovered. Next, let us test in (\ref{l1:vf2}) with $v\in C^{\infty}_c(\Omega).$ Since $\gamma v(\bx)=0$ on $\Gamma$ then, by (\ref{l1:vf2}) we obtain \[a(u,v)=\int_{\Omega} f_0 \, v\,dx.\] And from this, keeping in mind (\ref{a}), by using the First Green's Formula and the Fundamental Lemma in the Calculus of Variations  we are led to (\ref{c1}). By standard arguments we are also lead to (\ref{c2}).
The pointwise relations (\ref{c1})-(\ref{c2}) are helpful for some error analysis of numerical schemes, see e.g. pages 170-172 in \cite{HS02}; if $\Gamma_2=\varnothing$ the analysis is easier.

What about the frictional contact conditions (\ref{c3})?

By means of (\ref{ecuy}), (\ref{ecuyy}), (\ref{a}), (\ref{f}), the First Green's Formula and (\ref{c1})-(\ref{c2}) we obtain
\[\langle \bar{\lambda},v\rangle_{X',X}=\langle \lambda,\gamma v\rangle_{Y,M}=-\int_{\Gamma_3}\xi\frac{\partial u}{\partial \nu}(\bx)\gamma v(\bx)\,d\Gamma\quad\mbox{ for all }v\in X.\]
Notice that everywhere in this section $u=\gamma u$ on $\Gamma_3$ due to the smoothness assumption.
Keeping in mind (\ref{ecuy}) we can write
\[-\int_{\Gamma_3}\xi \frac{\partial u}{\partial \nu}(\bx)\gamma v(\bx)\leq \int_{\Gamma_3}g|\gamma v(\bx)|\,d\Gamma\quad\mbox{ for all }v\in X.\]
Moreover, using (\ref{i3}) we have
\[-\int_{\Gamma_3}\xi \frac{\partial u}{\partial \nu}(\bx) u(\bx)= \int_{\Gamma_3}g| u(\bx)|\,d\Gamma.\]
Therefore, for all $v\in X,$
\[\int_{\Gamma_3}\xi \frac{\partial u}{\partial \nu}(\bx)(\gamma v(\bx)- u(\bx))\,d\Gamma\geq -\int_{\Gamma_3}g(|\gamma v(\bx)|-|u(\bx)|)\,d\Gamma.\]

Below, we will test with functions
$v=\pm t w+u,$ where $t\in (0,\infty)$ and $w\in X$ such that  $\gamma w$ is smooth enough, with compact support on $\Gamma_3,$ and $\gamma w\geq 0$ a.e. on $\Gamma_3.$

If $v=tw+u$ then
\begin{equation}\label{lll}
\int_{\Gamma_3}\xi \frac{\partial u}{\partial \nu}(\bx)\gamma w(\bx)\,d\Gamma\geq -g\int_{\Gamma_3}\frac{|\gamma (tw+u)(\bx)|-|\gamma u(\bx)|}{t}\,d\Gamma.
\end{equation}
We observe that, a.e. on $\Gamma_3,$
$$\lim_{t\downarrow 0}\frac{|\gamma (tw+u)(\bx)|-|\gamma u(\bx)|}{t}=
\left\{
\begin{array}{ll}
\gamma w(\bx) &\mbox{ if }  u(\bx)=0;\\
\gamma w(\bx)\operatorname{sgn}(u(\bx))&\mbox{ if } u(\bx)\neq 0.
\end{array}\right.
$$
 Let us introduce $F_+:\mathbb R\to \mathbb R,$
 $$F_+(r)=
\left\{
\begin{array}{cl}
1 &\mbox{ if }  r=0;\\
\operatorname{sgn}(r)&\mbox{ if } r\neq 0.
\end{array}\right.
$$
Due to the  Lebesgue's dominated convergence, after we pass to the limit in (\ref{lll}), as $t\to 0$, we obtain
\begin{equation*}
\int_{\Gamma_3} \left(\xi\frac{\partial u}{\partial \nu}(\bx)+g\,F_+(u(\bx))\right)\gamma w(\bx)\,d\Gamma\geq 0.
\end{equation*}
 And then, by classical arguments of mathematical analysis, we deduce that, a.e. on $\Gamma_3,$
 \[\xi\frac{\partial u}{\partial \nu}(\bx)+g\,F_+(u(\bx))\geq 0\]
 which yields
\begin{eqnarray}
\xi \frac{\partial  u}{\partial \nu}(\bx)+g\frac{u(\bx)}{| u(\bx)|}\geq 0\mbox{ if } u(\bx)\neq 0;\label{I}\\
\xi \frac{\partial u}{\partial \nu}(\bx)+g\geq 0\mbox{ if } u(\bx)=0.\label{II}
\end{eqnarray}

If $v=-tw+u,$ then
\begin{equation*}\label{llll}
\int_{\Gamma_3}\xi \frac{\partial u}{\partial \nu}(\bx)\gamma w(\bx)\,d\Gamma\leq g\int_{\Gamma_3}\frac{|\gamma (-tw+u)(\bx)|-|\gamma u(\bx)|}{t}\,d\Gamma
\end{equation*}
and, a.e. on $\Gamma_3,$
$$\lim_{t\downarrow 0}\frac{|\gamma (-tw+u)(\bx)|-|u(\bx)|}{t}=
\left\{
\begin{array}{ll}
\gamma w(\bx) &\mbox{ if }  u(\bx)=0;\\
-\gamma w(\bx)\operatorname{sgn}(u(\bx))&\mbox{ if } u(\bx)\neq 0.
\end{array}\right.
$$
 Let us introduce $F_-:\mathbb R\to \mathbb R,$
 $$F_-(r)=
\left\{
\begin{array}{cl}
1 &\mbox{ if }  r=0;\\
-\operatorname{sgn}(r)&\mbox{ if } r\neq 0.
\end{array}\right.
$$
Therefore,
\begin{equation*}
\int_{\Gamma_3}\left(\xi \frac{\partial u}{\partial \nu}(\bx)-g\,F_-(u(\bx))\right)\gamma w(\bx)\,d\Gamma\leq 0.
\end{equation*}
Thus, a.e. on $\Gamma_3,$
\begin{eqnarray}
\xi \frac{\partial  u}{\partial \nu}(\bx)+g\frac{ u(\bx)}{| u(\bx)|}\leq 0\mbox{ if } u(\bx)\neq 0;\\\label{III}
\xi \frac{\partial u}{\partial \nu}(\bx)-g\leq 0\mbox{ if } u(\bx)=0.\label{IV}
\end{eqnarray}
The frictional interaction condition (\ref{c3}) is now a consequence of (\ref{I})-(\ref{IV}).

The regularity results remain open.

\vskip 10mm

\noindent\textbf{Acknowledgements} This project has received funding from the European Union's Horizon 2020
Research and Innovation Programme under the Marie Sklodowska-Curie
Grant Agreement No 823731 CONMECH.

\end{document}